\documentclass{amsart}
\usepackage{amsmath,amssymb,latexsym,soul,cite}
\usepackage[dvipsnames]{xcolor}
\usepackage{color,enumitem,graphicx}
\usepackage[colorlinks=true,urlcolor=blue,
citecolor=red,linkcolor=blue,linktocpage,pdfpagelabels,
bookmarksnumbered,bookmarksopen]{hyperref}
\usepackage[utf8]{inputenc}
\usepackage[hyperpageref]{backref}

\usepackage[left=2.8cm,right=2.8cm,top=2.9cm,bottom=2.9cm]{geometry}

\numberwithin{equation}{section}

\newtheorem{thm}{Theorem}[section]
  \theoremstyle{plain}
  \newtheorem{lem}[thm]{Lemma}
  \theoremstyle{plain}
  
  \theoremstyle{plain}
  
  \theoremstyle{plain}
  
    \theoremstyle{definition}

\newcommand{\R}{{\mathbb R}}
\newcommand{\N}{{\mathbb N}}

\title[]{The buckling eigenvalue problem in the annulus}

\author[D. Buoso]{Davide Buoso}
 \author[E. Parini]{Enea Parini}

\address[D.\ Buoso]{\'Ecole Polytechnique F\'ed\'erale de Lausanne, MATH SCI-SB-JS 
\newline\indent
Station 8, CH-1015 Lausanne, Switzerland}
\email{davide.buoso@epfl.ch}

\address[E.\ Parini]{Aix-Marseille Universit\'e, CNRS, Centrale Marseille, I2M, UMR 7373
\newline\indent
39 Rue Frederic Joliot Curie, 13453 Marseille, France}
\email{enea.parini@univ-amu.fr}


\thanks{}

\begin{document}

\begin{abstract}
We consider the buckling eigenvalue problem for a clamped plate in the annulus. We identify the first eigenvalue in dependence of the inner radius, and study the number of nodal domains of the corresponding eigenfunctions. Moreover, in order to investigate the asymptotic behavior of eigenvalues and eigenfunctions as the inner radius approaches the outer one, we provide an analytical study of the buckling problem in rectangles with mixed boundary conditions.
\end{abstract}

\maketitle


\section{Introduction}

If we consider the eigenvalue problem for the Dirichlet Laplacian on a bounded domain $\Omega\subset\mathbb{R}^N$,
\begin{equation}
\label{laplacian}
\begin{cases}
-\Delta u  = \lambda u, & \text{in }\Omega, \\
u  = 0, & \text{on }\partial \Omega,
\end{cases} 
\end{equation}
the first eigenvalue $\lambda_1^D(\Omega)$, differently from higher ones, exhibits some remarkable properties. Notably, the first eigenvalue is simple, and every associated eigenfunction does not change sign in $\Omega$. Higher order eigenvalue problems, on the other hand, present a different situation. For instance, in the case of the Dirichlet Bilaplacian
\begin{equation}
\label{bilaplacian}
\begin{cases}
\Delta^2 u  = \lambda u, & \text{in }\Omega, \\
u  = \partial_\nu u = 0, & \text{on }\partial \Omega,
\end{cases} 
\end{equation}
where $\nu$ denotes the unit outer normal, the first eigenfunction might be sign-changing. This is the case when $\Omega$ is a square \cite{coffmansquare}, or an elongated ellipse \cite{garabedian}. More generally, if the domain has corners, then all the eigenfunctions oscillate in the corner (see \cite{codu,codu2,kkm}). In addition, the first eigenvalue might even be multiple: this is indeed the case for annuli with small inner radii (see \cite{coffmanduffinshaffer}).

Strangely enough, the buckling eigenvalue problem
\begin{equation}
\label{bucklingequation}
\begin{cases}
\Delta^2 u = -\lambda \Delta u, & \text{in }\Omega, \\
u = \partial_\nu u=0, & \text{on }\partial \Omega,
\end{cases} 
\end{equation}
presents features that puts it in between the Dirichlet Laplacian and the Dirichlet Bilaplacian. In fact, while it could be thought of as a Laplacian applied on the space $\nabla H^2_0(\Omega)\subset(H^1(\Omega))^N$, and even the Weyl limit is the same as for problem \eqref{laplacian}, the arguments that are used to prove simplicity and positivity of the first eigenfunction do not apply here. Indeed, as for the Bilaplacian, all the eigenfunctions of the buckling problem oscillate in corners (see \cite{kkm}).

The fact that the first eigenfunction of higher order operators can be sign changing, as well as the lack of a maximum principle, has sparkled a lot of interest in the literature towards the understanding of how a solution can change sign or in which cases the so-called ``positivity preserving property'' holds or, equivalently, under which assumptions the Green function is positive. While for second order problems these questions are trivial because of the maximum principle and other important tools such as the Krein-Rutman theory, their higher order counterparts have shown to be extremely difficult to approach and, apart from the case of corners \cite{kkm}, only very specific cases have been successfully studied, usually for the clamped plate problem (we refer the reader to \cite{ggs} for an extensive discussion). Even estimates for the first eigenvalue on a rectangle turn out to be much more difficult to obtain than those for the Laplacian, see \cite{buoso-freitas, owen}. Regarding the buckling problem \eqref{bucklingequation}, even though the results in \cite{kkm} still apply, very little is known for shapes different from balls, where the spectrum and the eigenspaces can be completely characterized in terms of Bessel functions.

In this paper, we study the buckling eigenvalues on annuli with varying inner radius. While it is possible to write the eigenfunctions by means of Bessel functions and harmonic polynomials, the eigenvalues are solutions to complicated transcendental equations and therefore it is difficult to order them in an analytical way. For this reason, with the aid of the software Mathematica\texttrademark, we compute the lower eigenvalues in order to find the smallest one (i.e., the first one). In particular, since all eigenfunctions can be written as the product of a radial function and a spherical harmonic, this allows to give an estimate of how many nodal domains the first eigenfunction presents.

In \cite{coffmanduffinshaffer}, the authors show that the first eigenvalue of problem \eqref{bilaplacian} on an annulus is double and the corresponding eigenfunctions of problem \eqref{bilaplacian} has one nodal line corresponding to the diameter, if the inner radius is small enough. Conversely, when the inner radius increases, after the threshold, the first eigenvalue becomes simple and with a (strictly) positive eigenfunction. We discover that the behavior of the buckling problem \eqref{bucklingequation} is totally different, since the first eigenvalue appears to be always multiple, while the number of nodal domains of the corresponding eigenfunctions is increasing with respect to the inner radius. Even though this result at first may be unexpected when compared with the analogues for \eqref{laplacian} and \eqref{bilaplacian}, we remark that higher order problems present a great variety of properties that, in many situations, have yet to be unveiled.

When the annulus shrinks towards the circumference, we observe that the number of diametric nodal lines increases as well. In particular, as these nodal lines are proliferating, we obtain an increase in nodal domains that are annular sectors but close (in a suitable sense) to rectangles. It makes sense then to study the buckling problem on a shrinking rectangle, with Dirichlet boundary conditions ($u=\partial_\nu u=0$) on the long edges and Navier ones ($u=\Delta u=0$) on the shrinking edges. Even though the two problems are not equivalent, the buckling problem on the rectangle with mixed boundary conditions is much easier to study and can provide qualitative information on the behavior of the eigenfunctions on a shrinking annulus. We note that biharmonic problems on rectangles with mixed boundary conditions have been studied recently in \cite{fergaz}, in relation with the stability of suspension bridges. 

Another interesting limiting problem is provided by the punctured disk, which is obtained when the inner radius converges to zero. Differently from the Laplacian, the punctured disk is not equivalent to the full disk for Bilaplacian problems such as \eqref{bilaplacian} or \eqref{bucklingequation}. The spectrum in this limiting case was already studied for the Dirichlet Bilaplacian in \cite{coffmanduffinshaffer} numerically, and in \cite{codu2} analytically, by means of an intricate study of the eigenvalues and using tables of values of zeroes of Bessel functions. Indeed, while the first eigenfunction of the Dirichlet Bilaplacian on the disk contains the Bessel function $J_0$, the first eigenfunction on the punctured disk contains instead $J_1$. While the behavior of problem \eqref{bucklingequation} differs from that of \eqref{bilaplacian} when the annulus shrinks, at the other end the two coincide, even though the two first eigenfunctions are not related.

As problem \eqref{laplacian} can be used to model vibrating elastic membranes, also problems \eqref{bilaplacian} and \eqref{bucklingequation} are of relevance in the theory of mechanics of deformable solids, and they are derived from the so-called Kirchhoff-Love model of plates (see \cite{rayleigh}). In particular, the Dirichlet Bilaplacian \eqref{bilaplacian} is also called ``vibrating clamped plate problem'' as it models the vibration of a clamped plate, while the buckling problem \eqref{bucklingequation} is related with the critical tension that leads to the deformation of a clamped plate. Hence, the interest for these problems comes not only from the mathematical point of view, but also from applications. As a consequence, shedding light on difficult questions like the ones we consider in this paper is more than mathematical speculation, since it can provide new ideas for the several applications that make use of these problems.

The paper is organized as follows. In Section 2 we set up the problem and state some of the basic properties of the eigenvalues. In Section 3 we compute the eigenfunctions and the eigenvalues of the buckling problem \eqref{bucklingequation} on an annulus and on the punctured disk, while Section 4 is devoted to the numerical calculations. In Section 5 we focus on the buckling problem on a rectangle with mixed boundary conditions. Finally, in Section 6 we collect a number of observations and open questions.


\section{Preliminaries}
\label{sec2}

Let $\Omega \subset \R^N$ be a bounded domain (i.e., an open connected set). We say that $\lambda \in \R$ is an eigenvalue of the buckling problem \eqref{bucklingequation} if there exists a function $u \in H^2_0(\Omega) \setminus \{0\}$ (the eigenfunction) which satisfies the equation in the weak sense, namely,
\begin{equation*}
 \int_\Omega \Delta u \Delta v = \lambda \int_\Omega \nabla u \nabla v, \qquad \text{for every } v \in H^2_0(\Omega).
\end{equation*}
The buckling eigenvalue problem admits an increasing sequence of (strictly) positive eigenvalues of finite multiplicity diverging to infinity, 
\[ 0 < \lambda_1(\Omega) \leq \lambda_2(\Omega) \leq ... \leq \lambda_k(\Omega) \to +\infty, \qquad \text{as }k \to +\infty\],
(see e.g., \cite[Lemma 2.1]{buosolamberti}). Additionally, the $k$-th eigenvalue can be characterized variationally as
\[ \lambda_k(\Omega)= \min_{\substack{ V \subseteq H^2_0(\Omega) \\ \dim V=k}}\  \max_{u\in V\setminus \{0\}} \frac{\int_\Omega |\Delta u|^2}{\int_\Omega |\nabla u|^2}.\]
This characterization implies the monotonicity of $\lambda_k$ with respect to set inclusion:
\begin{equation}
\label{crescenza}
\Omega_1 \subset \Omega_2 \Rightarrow \lambda_k(\Omega_2) \leq \lambda_k(\Omega_1).
\end{equation}
Problem \eqref{bucklingequation} also enjoys the following scaling property: for $t>0$, if we set
\[ t\Omega :=\{ x \in \R^N \,|\, x/t \in \Omega\},\]
we have
\[ \lambda_k(t \Omega) = \frac{1}{t^2}\lambda_k(\Omega).\]
This implies in particular that the shape functional
\[ \Omega \mapsto \lambda_1(\Omega)|\Omega|^{\frac{2}{N}}\]
is scaling invariant.

When $\Omega=B_R \subset \R^2$ is a disk of radius $R$, the eigenfunctions have the form
\begin{equation}
\label{eigenball}
u(r,\theta)=\left(J_k(r\sqrt{\lambda})-\frac{J_k(R\sqrt{\lambda})}{R^k}r^k\right)e^{\pm i k \theta},
\end{equation}
where $(r,\theta)$ are polar coordinates, and the corresponding eigenvalues are
\[
\lambda=\left(\frac{j_{k+1,t}}{R}\right)^2,
\]
for some $t\in\mathbb N^*$, where $j_{\nu,t}$ is the $t$-th zero of the Bessel function $J_\nu$.

In particular, for $R=1$ the first eigenvalue is \[ \lambda_1(B_1)=j_{1,1}^2 \simeq 14.6819,\]
and the associated eigenfunctions are radially symmetric and do not change sign in $B_1$.
It is also worth mentioning that, for convex planar domains, Payne was able to give sharp estimates of $\lambda_1(\Omega)$ in terms of the first two eigenvalues $\lambda_1^D(\Omega)$ and $\lambda_2^D(\Omega)$ of the Dirichlet Laplacian \eqref{laplacian}:
\[ \lambda_2^D(\Omega) \leq \lambda_1(\Omega) \leq 4 \lambda_1^D(\Omega),\]
where the first equality holds true when $\Omega$ is a disk, and the second one holds true in the limit when $\Omega$ is an infinite strip (see \cite{payne}).


\section{The annulus and the punctured disk}

We turn now to the computation of the solutions of the buckling problem \eqref{bucklingequation} in the annulus
\[ \Omega_a := \{x \in \R^2 \,|\,a<|x|<1\},\]
with $a \in (0,1)$. It is classical to show that eigenfunctions are of the form
\begin{equation}
\label{eigenfct}
 u(r,\theta) = v(r)e^{\pm ik\theta}, \qquad k \in \mathbb{N}.
\end{equation}
We refer for instance to \cite[Section V.5]{couranthilbert} for the case of the Laplacian. Set $\mu = \sqrt{\lambda}$. As in \cite[Lemma 2.1]{decoster}, we have that, if $k=0$,
\begin{equation*}
v(r) = AJ_0(\mu r) + B Y_0(\mu r) + C + D \ln{r},
\end{equation*}
while if $k \geq 1$,
\begin{equation}
\label{vk}
v(r) = AJ_k(\mu r) + B Y_k(\mu r) + Cr^k + D r^{-k}.
\end{equation}
Here $Y_k$ is the Bessel function of the second kind of order $k$.

Using \eqref{eigenfct}, the boundary conditions in \eqref{bucklingequation} become
\[v(a),v'(a),v(1),v'(1)=0,\]
so that, in the case $k=0$, we obtain the following system
\[ \left\{\begin{array}{l}
AJ_0(\mu) + B Y_0(\mu) + C = 0 \\
A\mu J_0'(\mu) + B\mu Y_0'(\mu) + D = 0 \\
AJ_0(\mu a) + B Y_0(\mu a) + C + D \ln{a} = 0 \\
A\mu J_0'(\mu a) + B\mu Y_0'(\mu a)  + D \frac{1}{a} = 0           
\end{array} \right.\]
or, in matrix form,
\begin{equation}\label{matrice1}
\begin{pmatrix}
J_0(\mu) & Y_0(\mu) & 1 & 0 \\
\mu J_0'(\mu) & \mu Y_0'(\mu) & 0 &  1 \\
J_0(\mu a) & Y_0(\mu a) & 1 &  \ln{a}\\
\mu J_0'(\mu a) & \mu Y_0'(\mu a) & 0 &  \frac{1}{a}
\end{pmatrix}
\begin{pmatrix} A \\ B \\ C \\ D \end{pmatrix}
= \begin{pmatrix} 0 \\ 0 \\ 0\\ 0\end{pmatrix}.
\end{equation}

Since a solution to \eqref{matrice1} exists only if the determinant of the matrix is zero, we compute it and, after using several properties of Bessel functions (see e.g., \cite[Ch.\ 9]{abrste}), we obtain that it equals
\begin{multline}
\label{zeroes0}
\frac{4}{\pi a}+ \mu \left( J_0(\mu)Y_1(\mu a) - J_1(\mu a)Y_0(\mu)\right) - \mu^2 \ln{a} \left(J_1(\mu)Y_1(\mu a)-J_1(\mu a)Y_1(\mu)\right)\\
+\frac{\mu}{a}\left(J_0(\mu a)Y_1(\mu)-J_1(\mu)Y_0(\mu a)\right).
\end{multline}

On the other hand, in the case $k \geq 1$, the boundary conditions yield
\[ \left\{\begin{array}{l}
AJ_k(\mu) + B Y_k(\mu) + C + D = 0 \\
A\mu J_k'(\mu) + B\mu Y_k'(\mu) + Ck - Dk = 0 \\
AJ_k(\mu a) + B Y_k(\mu a) + C a^k + D a^{-k} = 0 \\
A\mu J_k'(\mu a) + B\mu Y_k'(\mu a)  + Ck a^{k-1} - Dk a^{-k-1} = 0           
\end{array} \right.\]
or, in matrix form,
\begin{equation}\label{matrice2}
\begin{pmatrix}
J_k(\mu) & Y_k(\mu) & 1 & 1 \\
\mu J_k'(\mu) & \mu Y_k'(\mu) & k &  -k \\
J_k(\mu a) & Y_k(\mu a) & a^k &  a^{-k}\\
\mu J_k'(\mu a) & \mu Y_k'(\mu a) & ka^{k-1} &  -ka^{-k-1}
\end{pmatrix}
\begin{pmatrix} A \\ B \\ C \\ D \end{pmatrix}
= \begin{pmatrix} 0 \\ 0 \\ 0\\ 0\end{pmatrix}.
\end{equation}

Again, since a solution to \eqref{matrice2} exists only if the determinant of the matrix is zero, we compute it and, after using several properties of Bessel functions, we obtain that it equals
\begin{multline}
\label{zeroesk}
\mu^2 \left(a^k-\frac{1}{a^k}\right)\left(J_{k-1}(\mu)Y_{k-1}(\mu a)- J_{k-1}(\mu a)Y_{k-1}(\mu)\right)- \frac{8k}{\pi a}\\ 
+2k\mu a^{k-1}\left(J_{k}(\mu a)Y_{k-1}(\mu)-J_{k-1}(\mu)Y_{k}(\mu a)\right) +\frac{2k\mu}{a^k}\left(J_k(\mu)Y_{k-1}(\mu a)-J_{k-1}(\mu a)Y_{k}(\mu)\right) .
\end{multline}

We observe that, while in a ball the ordering of the eigenvalues is directly linked to the ordering of zeroes of Bessel functions (cf.\ Section \ref{sec2}), the case of the annulus is much more involved as we would need to understand how the zeroes of the functions \eqref{zeroes0} and \eqref{zeroesk} interlace. Therefore, for the annulus the difficulty of ordering the eigenvalues is amplified, and we will do it numerically in Section \ref{sec3}.

If we set $a=0$, $\Omega_0$ is then a punctured disk, and it is possible to set the buckling problem \eqref{bucklingequation} in $\Omega_0$. In this case, though, problem \eqref{bucklingequation} has to be interpreted in a slightly different way.
We first observe that, since a point in $\R^2$ has positive $H^2$-capacity, but zero $H^1$-capacity, functions in $H^2_0(\Omega_0)$ must equal zero at the origin, but their gradient need not vanish. Nevertheless, the gradient of any eigenfunction is bounded in $\Omega_0$. This follows from the fact that eigenfunctions are of the form \eqref{eigenfct} and that their gradient, as a function of $H^1(\Omega_0)^2$, is absolutely continuous on almost every line passing through the origin. This readily implies that the boundary conditions for the eigenfunctions translate as
$$
v(0),v(1),v'(1)=0,\ \ v'(0)\in\mathbb{R}.
$$
We mention that the boundedness of the gradient of eigenfunctions in $\Omega_0$ can also be deduced from a more general result by Mayboroda and Maz'ya \cite[Proposition 8.1]{mayborodamazya}.

In order to compute the eigenfunctions and the eigenvalues, we need to recall the asymptotic behaviors of Bessel functions $Y$ around zero (cf.\ \cite[formulas 9.1.11 and 9.1.13]{abrste}):
\begin{equation*}
Y_0(z)=\frac 2 \pi \left(\ln{\frac z 2}+\gamma_{\rm EM}\right)J_0(z)+ P_0(z),
\end{equation*}
where $\gamma_{\rm EM}$ is the Euler-Mascheroni constant, and
\begin{equation}
\label{besselypropk}
Y_k(z)=-\frac{2^k}{\pi z^k}\sum_{n=0}^{k-1}\frac{(k-n-1)!}{n!}\left(\frac z 2\right)^{2n}+\frac 2 \pi J_k(z)\ln{\frac z 2} -\frac{z^k}{2^k\pi} P_k(z),
\end{equation}
where $P_k(z)$ is a power series in $z^2$ for any $k\ge 0$. In particular, we use \eqref{besselypropk} to show that, for $k \geq 1$, the coefficients of $Y_k(r)$ and $r^{-k}$ in \eqref{vk} must vanish.

For $k \geq 3$, this can be easily seen due to the appearance of a singular function of order $r^{2-k}$ in \eqref{besselypropk} which can not be canceled out. 
For $k=2$, the singularity of $Y_2$ at the origin implies the condition $D=\frac{4}{\pi}B$; but since
\[ \lim_{r \to 0} v(r) = \lim_{r \to 0} B\left[Y_2(r)+\frac{4}{\pi}r^{-2}\right] = -\frac{B}{\pi},\]
it follows that $B=D=0$. Finally, for $k=1$, the behavior of $Y_1$ requires $D=\frac{2}{\pi}B$; but the presence of the function $r \mapsto J_1(r)\ln\left(\frac{r}{2}\right)$ in the asymptotic development of $Y_1$ implies that $B$ must be zero in order to guarantee the boundedness of $v'$.
Due to the previous considerations, for $k \geq 1$ it must hold
$$
v(r)=J_k(\mu r)-J_k(\mu)r^k,
$$
and the corresponding eigenvalues equal those of the full disk. For $k=0$, we obtain instead the following system
\[ \left\{\begin{array}{l}
AJ_0(\mu) + B Y_0(\mu) + C = 0 \\
A\mu J_0'(\mu) + B\mu Y_0'(\mu) + D = 0 \\
A + B \frac 2 \pi \left(\ln{\frac \mu 2} +\gamma_{\rm EM} \right) + C = 0 \\
B \frac 2 \pi   + D = 0           
\end{array} \right.\]
or, in matrix form,
\begin{equation}\label{matrice3}
\begin{pmatrix}
J_0(\mu) & Y_0(\mu) & 1 & 0 \\
\mu J_0'(\mu) & \mu Y_0'(\mu) & 0 &  1 \\
1 & \frac 2 \pi \left(\ln{\frac \mu 2} +\gamma_{\rm EM} \right) & 1 &  0 \\
0 & \frac 2 \pi & 0 &  1
\end{pmatrix}
\begin{pmatrix} A \\ B \\ C \\ D \end{pmatrix}
= \begin{pmatrix} 0 \\ 0 \\ 0\\ 0\end{pmatrix}.
\end{equation}

Once more, since a solution to \eqref{matrice3} exists only if the determinant of the matrix is zero, we compute it and, after using several properties of Bessel functions, we obtain that it equals
\begin{equation}
\label{zeroespunctured}
\frac 2 \pi \left(J_0(\mu)-2\right) +\frac {2\mu}\pi J_1(\mu)\left(\ln{\frac \mu 2}+\gamma_{\rm EM}\right)-\mu Y_1(\mu).
\end{equation}

The first nontrivial zero of \eqref{zeroespunctured} is $\mu = 6.6478167$. By the considerations above, we deduce that the first eigenvalue of $\Omega_0$ is given by
\[ \lambda_1(\Omega_0) = j_{2,1}^2 \simeq 23.3746,\]
and any associated eigenfunction has exactly two nodal domains, with a diametrical nodal line.


\section{Numerics}
\label{sec3}

In this section we analyze how the first eigenvalue changes with respect to the inner radius $a$ of the annulus $\Omega_a$.

In the following we denote by $\tau_k(a)$, for $k \in \N$, the smallest eigenvalue of \eqref{bucklingequation} in $\Omega_a$, associated with eigenfunctions of the form
\[ u(r,\theta) = v(r)e^{\pm ik\theta}.\]
It is easy to see that, for fixed $k$, the function
\[ a \in [0,1) \to \tau_k(a) \]
is monotone increasing.

For some particular values of $a$, we determined the value of $k_{\rm opt}$ such that $\tau_{k_{\rm opt}}(a)=\lambda_1(\Omega_a)$ by means of the software Mathematica\texttrademark. In Table \ref{tabellanumerica}, the index $k_{\rm max}$ denotes the maximum number of cases we had to check, and it is determined by means of the following algorithm:
\begin{itemize}
 \item[(i)] Set $k=0$, $k_{\rm opt}=0$, and compute $\tau_k(a)$. 
 \item[(ii)] Let $k_{\rm max} \in \N$ be the smallest index such that $\tau_{k_{\rm max}}(0)> \tau_{k_{\rm opt}}(a)$. Such an index is well defined, since $\tau_k(0)=j_{k+1,1}^2$, where $j_{k+1,1}$ is the smallest zero of the Bessel function $J_{k+1}$, and by monotonicity of the map $a \mapsto \tau_k(a)$.
 \item[(iii)] If $k \geq k_{\rm max}$, stop the algorithm, and return $\lambda_1(\Omega_a)=\tau_{k_{\rm opt}}(a)$.
 \item[(iv)] If $k \leq k_{\rm max} -1$, set $k=k+1$, and compute $\tau_k(a)$. If $\tau_k(a) < \tau_{k_{\rm opt}}(a)$, set $k_{\rm opt}=k$. Return to (ii).
\end{itemize}

\begin{table}[!ht] \label{tabellanumerica}
\begin{center}
\begin{tabular}{|c|c|c|c|c|c|}
\hline
$a$ & $k_{\rm max}$ & $k_{\rm opt}$ & $\sqrt{\tau_{k_{\rm opt}}(a)} = \sqrt{\lambda_1(\Omega_a)}$&  $\lambda_1(\Omega_a)|\Omega_a|$ \\
\hline
$0.00$ & $1$ & $1$ & $5.13562$&  $82.858$ \\
\hline
$0.05$ & $2$ & $1$ & $6.23824$&  $121.95$ \\
\hline
$0.10$ & $3$ & $1$ & $6.71001$&  $140.03$ \\
\hline
$0.15$ & $3$ & $2$ & $7.06409$&  $153.24$ \\
\hline
$0.20$ & $3$ & $2$ & $7.50246$& $169.76$  \\
\hline
$0.25$ & $4$ & $2$ & $8.02527$&  $189.69$ \\
\hline
$0.30$ & $4$ & $2$ & $8.63688$&  $213.26$ \\
\hline
$0.35$ & $5$ & $3$ & $9.34321$&  $240.65$ \\
\hline
$0.40$ & $6$ & $3$ & $10.0995$&  $269.17$ \\
\hline
$0.45$ & $6$ & $3$ & $11.0318$&  $304.91$ \\
\hline
$0.50$ & $7$ & $4$ & $12.1553$&  $348.13$ \\
\hline
$0.55$ & $9$ & $4$ & $13.5034$&  $399.56$ \\
\hline
$0.60$ & $10$ & $5$ & $15.2003$& $464.55$ \\
\hline
$0.65$ & $11$ & $6$ & $17.3833$&  $548.23$ \\
\hline
$0.70$ & $15$ & $7$ & $20.2830$&  $659.15$\\
\hline
$0.75$ & $19$ & $8$ & $24.3501$& $814.95$\\
\hline
$0.80$ & $24$ & $11$ & $30.4382$&  $1047.8$  \\
\hline
$0.82$ & $27$ & $12$ & $33.8219$&  $1177.3$ \\
\hline
$0.84$ & $31$ & $14$ & $38.0521$&  $1339.2$ \\
\hline
$0.86$ & $36$ & $16$ & $43.4894$&  $1547.2$ \\
\hline
$0.88$ & $43$ & $19$ & $50.7402$&  $1824.7$ \\
\hline
$0.90$ & $53$ & $23$ & $60.8901$&  $2213.1$ \\
\hline
$0.91$ & $60$ & $25$ & $67.6582$&  $2472.1$ \\
\hline
$0.92$ & $68$ & $29$ & $76.1145$&  $2795.6$ \\
\hline
$0.93$ & $78$ & $33$ & $86.9885$&  $3211.7$ \\
\hline
$0.94$ & $92$ & $39$ & $101.488$&  $3766.5$ \\
\hline
$0.95$ & $112$ & $47$ & $121.786$&  $4543.1$ \\
\hline
$0.96$ & $142$ & $59$ & $152.234$&  $5708.1$  \\
\hline
$0.97$ & $192$ & $79$ & $202.979$&  $7649.6$  \\
\hline
$0.98$ & $292$ & $119$ & $304.469$&  $11533$  \\
\hline
$0.99$ & $593$ & $239$ & $608.940$&  $23182$  \\
\hline
$0.995$ & $1198$ & $479$ & $1217.880$&  $46481$ \\
\hline
\end{tabular}
\end{center}
\caption{The first eigenvalues and related quantities for various values of $a$.}
\end{table}

We observe that, despite the roughness in estimating $k_{\rm max}$ in (ii), especially when $a$ is close to $1$, it turns out that the algorithm is not too computationally demanding. Moreover, although the index $k_{\rm max}$ grows faster than $k_{\rm opt}$, even for $a=0.995$, we still have that $k_{\rm max}\le 3k_{\rm opt}$.

Table \ref{tabellanumerica} has to be complemented with the information of the unit ball $B_1$. In that case, we have no $k_{\rm max}$, and $k_{\rm opt}=0$. As expected from \eqref{crescenza}, the value for $\sqrt{\lambda_1(B_1)}$ is
$$
\sqrt{\lambda_1(B_1)}=j_{1,1}\simeq 3.8317,
$$
the lowest of all, and moreover,
$$
\lambda_1|B_1|\simeq 12.038,
$$
again, the lowest of all. It is worth remarking that, in the case of problem \eqref{laplacian}, the well-known Faber-Krahn inequality tells us that the quantity $\lambda_1^D(\Omega)|\Omega|$ has the ball as the only minimizer (we refer to \cite[Section 3.2]{henrot} and to \cite{brascodephilippisvelichkov} where the stability issue is also considered), while a similar statement holds for the Dirichlet Bilaplacian \eqref{bilaplacian} as well (we refer to \cite{ashbaughbenguria, nadirashvili}). Unfortunately, this is still a conjecture for the buckling problem \eqref{bucklingequation}, known as Szego's conjecture. For more details on this subject, we refer the reader to \cite[Section 11.3]{henrot} and the references therein (see also the recent works \cite{abf,bcp,buoso} for shape optimization results on biharmonic eigenvalues).

In addition, the information that $\lambda_1(\Omega_a)$ is increasing in $a$ is of no surprise in view of \eqref{crescenza}, indeed it is expected to diverge to infinity for $a\to1^-$. The interesting fact is that also $\lambda_1(\Omega_a)|\Omega_a|$ is increasing in $a$, providing evidence for the validity of Szego's conjecture.

An interesting question is to determine the constants for the asymptotic behaviors

\[ k_{opt}(a) \sim \frac{c_k}{1-a}, \qquad \sqrt{\tau_{k_{\rm opt}}(a)} \sim \frac{c_\mu}{1-a} \qquad \text{as } a \to 1^-.\]

Our numerical computations support the claim that $c_k \simeq 2.38$, and $c_\mu \simeq 6.0894$, yet a conclusive argument does not seem to be easily achievable.


Moreover, we observe that, for any $a$, the value of $\mu_k(a)$ is strictly decreasing in $k$ for $1\le k\le k_{\rm opt}$, and strictly increasing for $k\ge k_{\rm opt}$. In other words, given the particular ordering of the analytical branches at $a=0$, at the other limit $a\to1^-$ they all seem to tend to switch and reverse the ordering. This can be already observed for the first five eigenvalues in Figure \ref{figura1}. This phenomenon is particularly interesting, and in open contrast to what happens in the case of the Dirichlet Bilaplacian \eqref{bilaplacian} (cf.\ \cite{coffmanduffinshaffer}).

\begin{figure}[h!]
\centering
\includegraphics[width=0.70\textwidth]{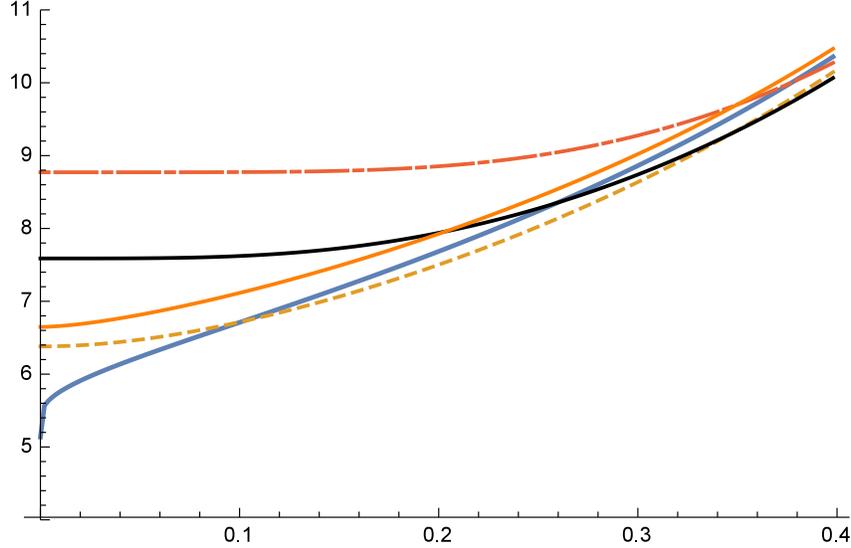}
\caption{The analytical branches of the first five eigenvalues for $a\in[0,0.4]$. 
}
\label{figura1}
\end{figure}

We conclude this section with the study of the positivity of the eigenfunctions. While all eigenfunctions associated with $k\ge 1$ are naturally sign-changing due to the angular part, it is still interesting to ask when the radial part has a definite sign. The eigenfunctions on the ball assume the form \eqref{eigenball} and it is easy to see that the eigenfunctions associated with eigenvalues $j_{k,1}^2$ have positive radial parts, using the oscillating behavior of the Bessel functions $J_k$. This behavior is not checked as easily in the case $a>0$, but still we conjecture that the eigenvalues given by the first zeroes of \eqref{zeroes0} and \eqref{zeroesk} have a positive radial part, as supported by some numerical computations we display in Figure \ref{figura2}.

\begin{figure}[h!]
\centering
\begin{tabular}{ccc}
\includegraphics[width=0.28\textwidth]{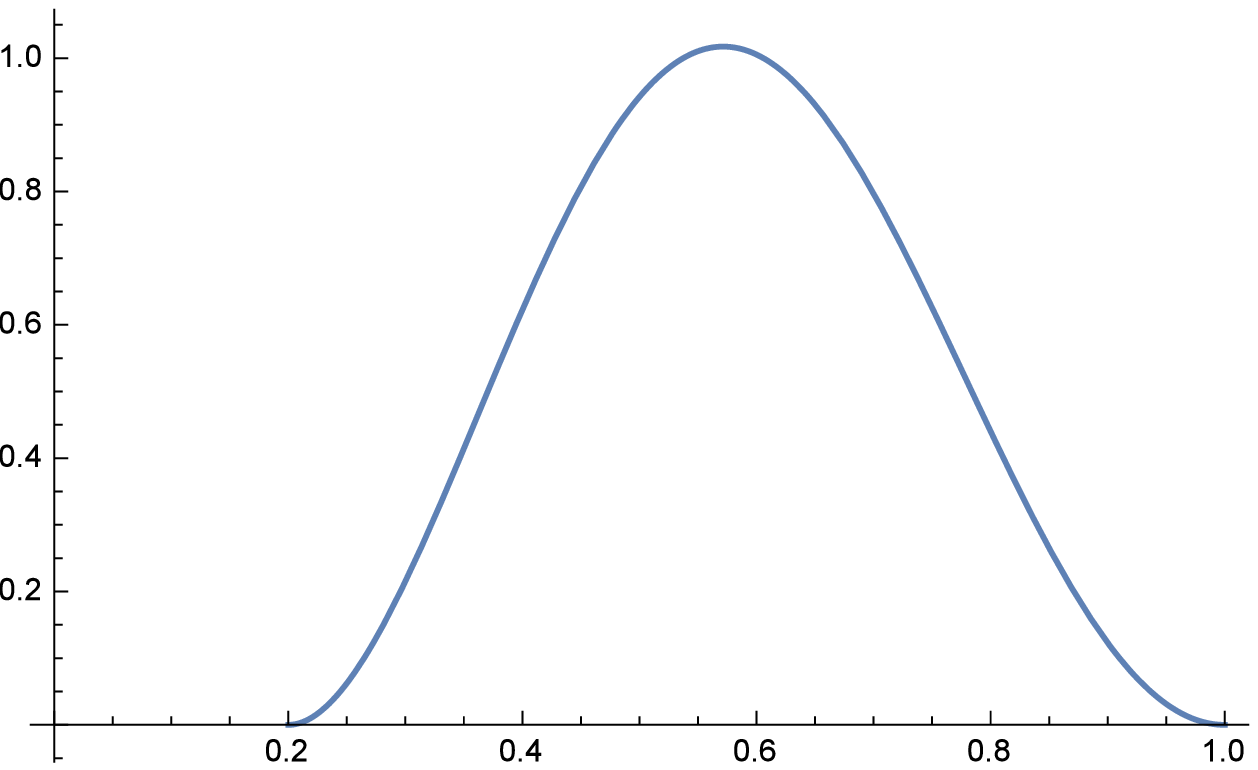} & \includegraphics[width=0.28\textwidth]{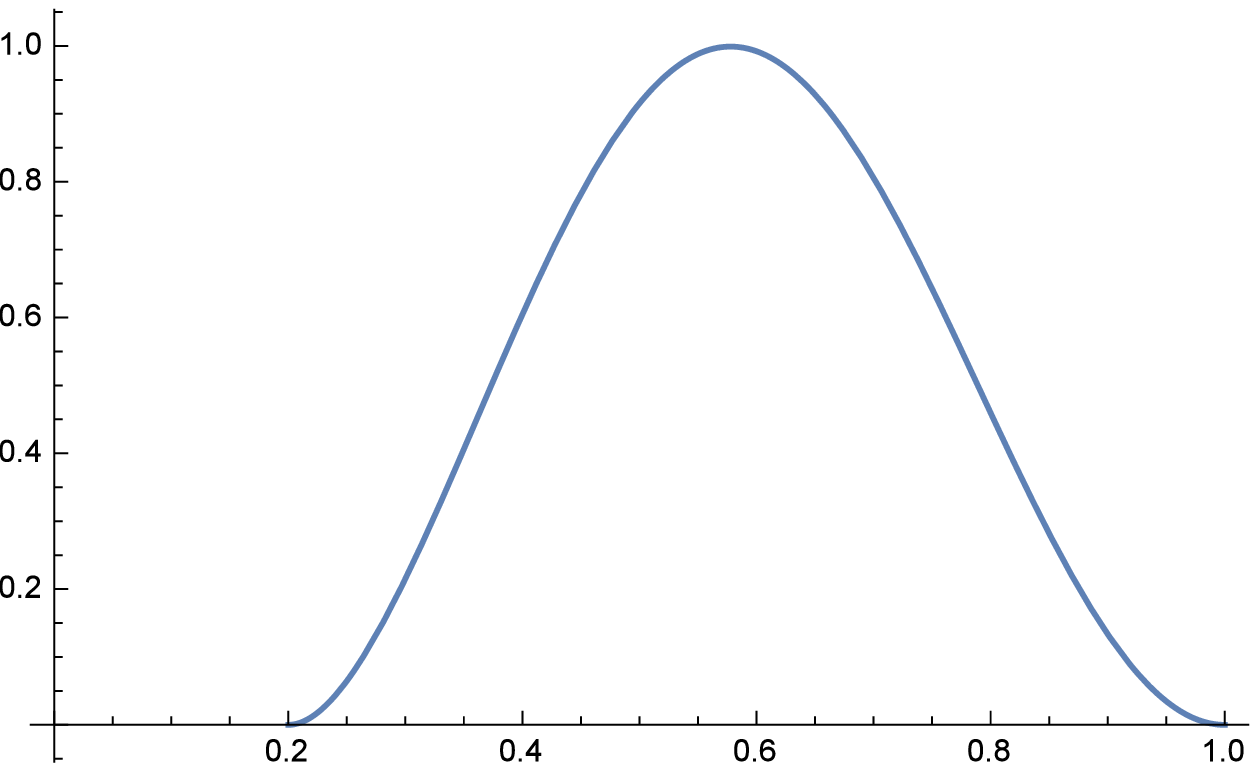} &
	\includegraphics[width=0.28\textwidth]{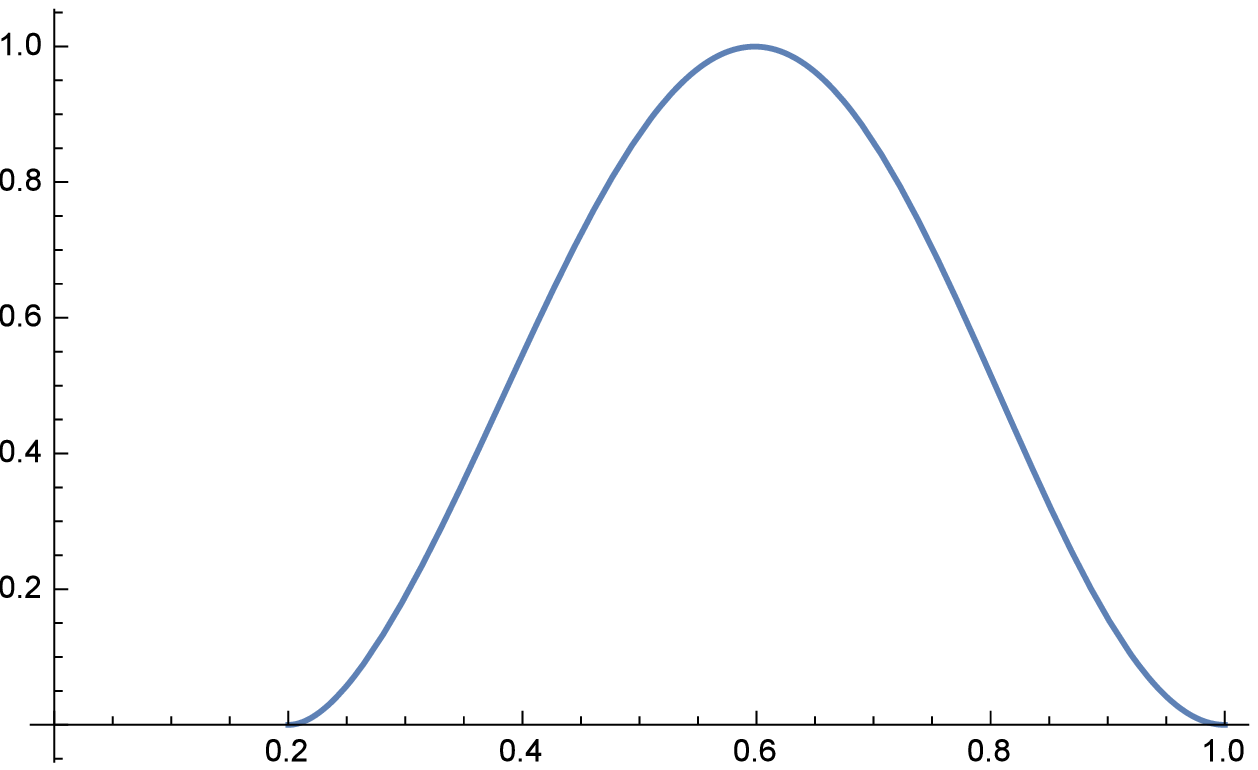}\\
	(a) $k=1$, $a=0.2$ & (b) $k=2$, $a=0.2$ & (c) $k=3$, $a=0.2$\\
\includegraphics[width=0.28\textwidth]{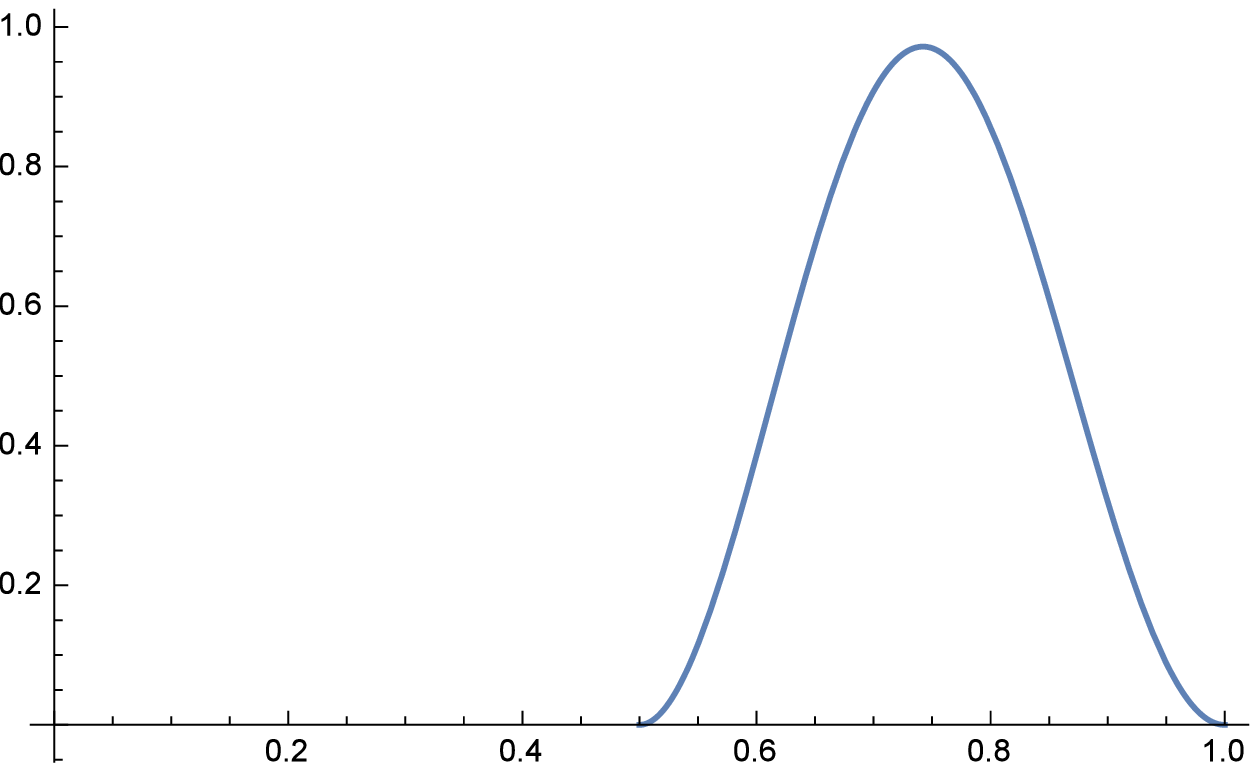} & \includegraphics[width=0.28\textwidth]{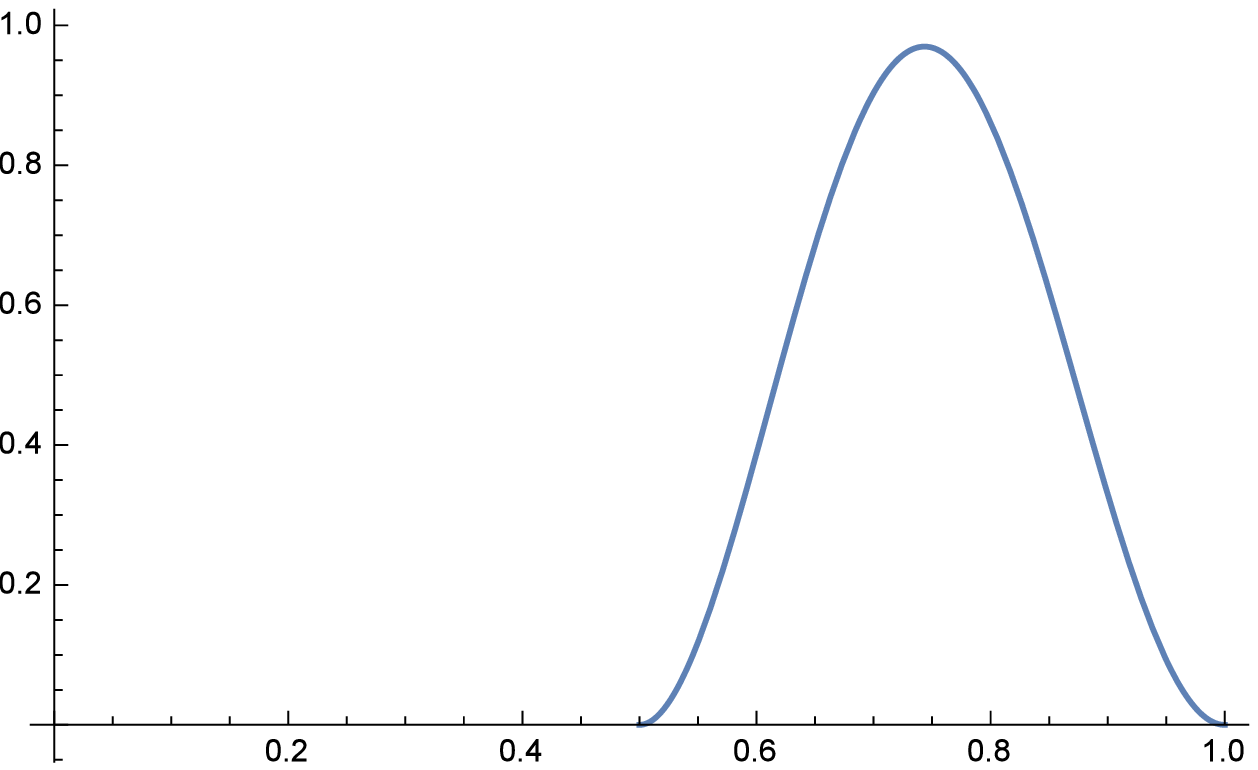} &
	\includegraphics[width=0.28\textwidth]{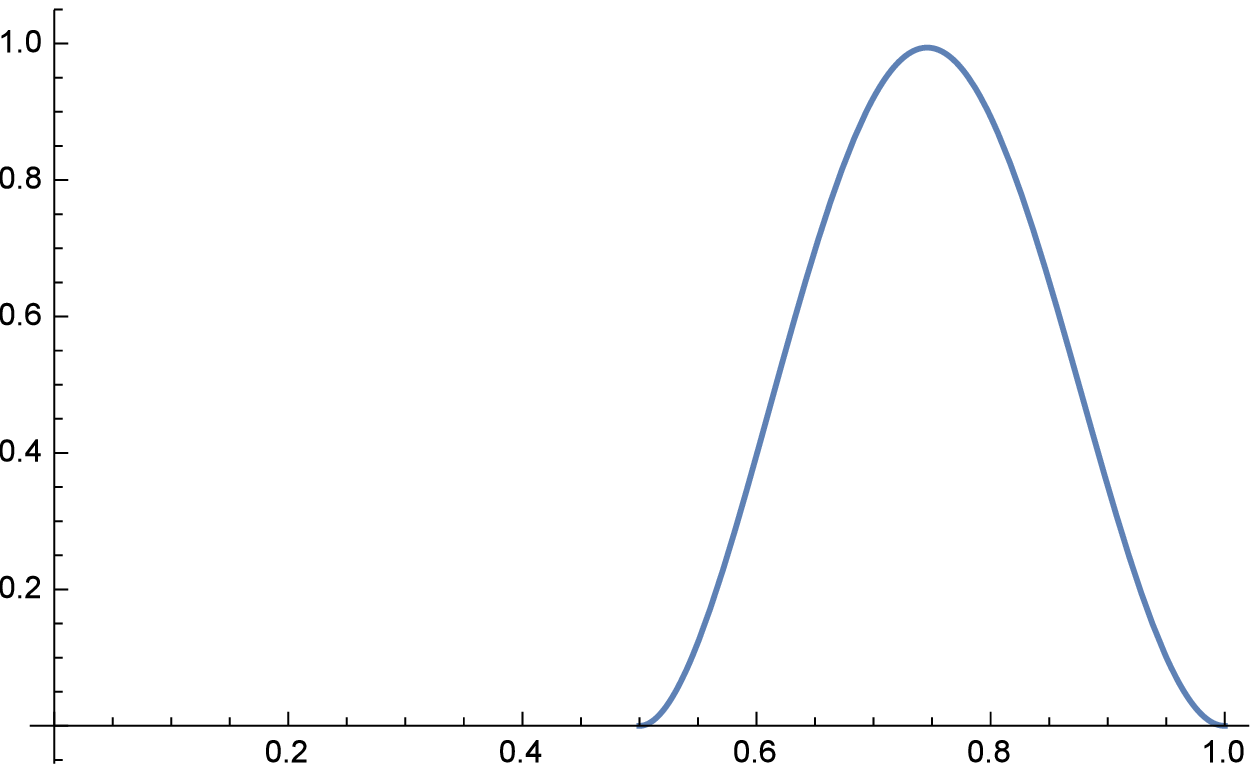}\\
	(d) $k=3$, $a=0.5$ & (e) $k=4$, $a=0.5$ & (f) $k=5$, $a=0.5$\\
\includegraphics[width=0.28\textwidth]{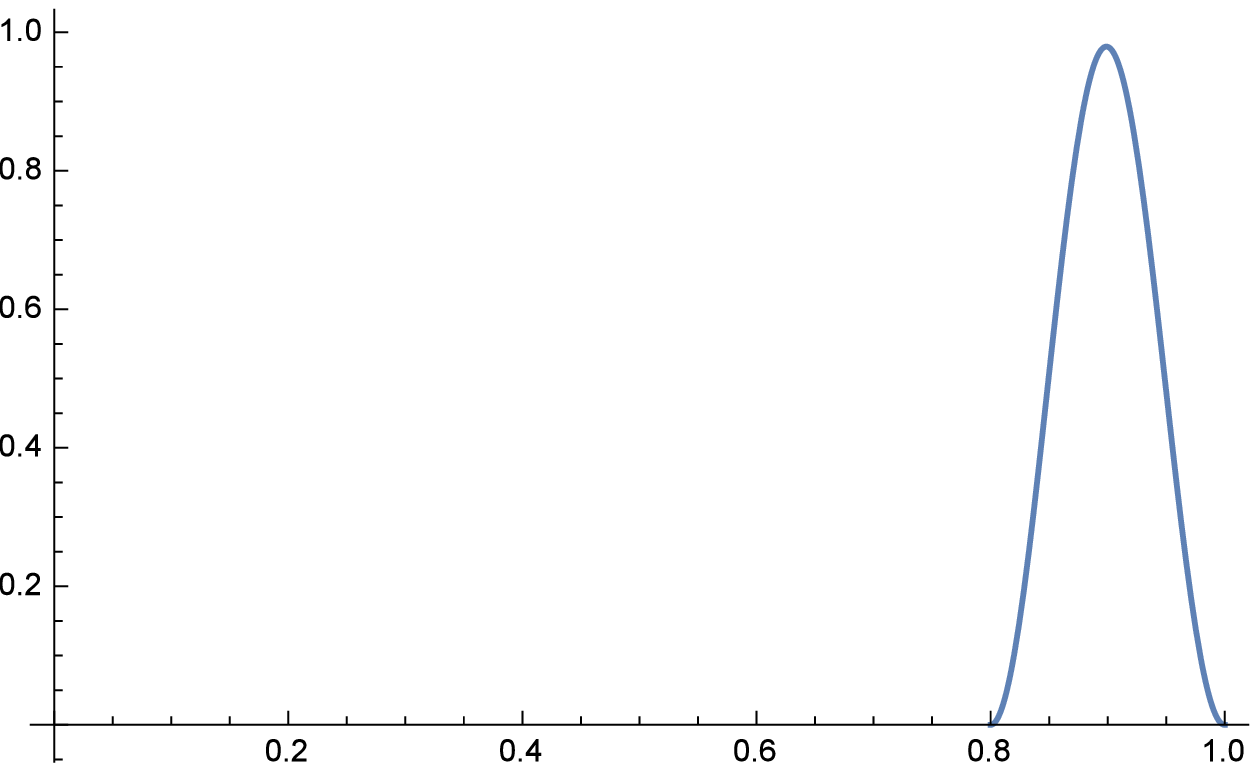} & \includegraphics[width=0.28\textwidth]{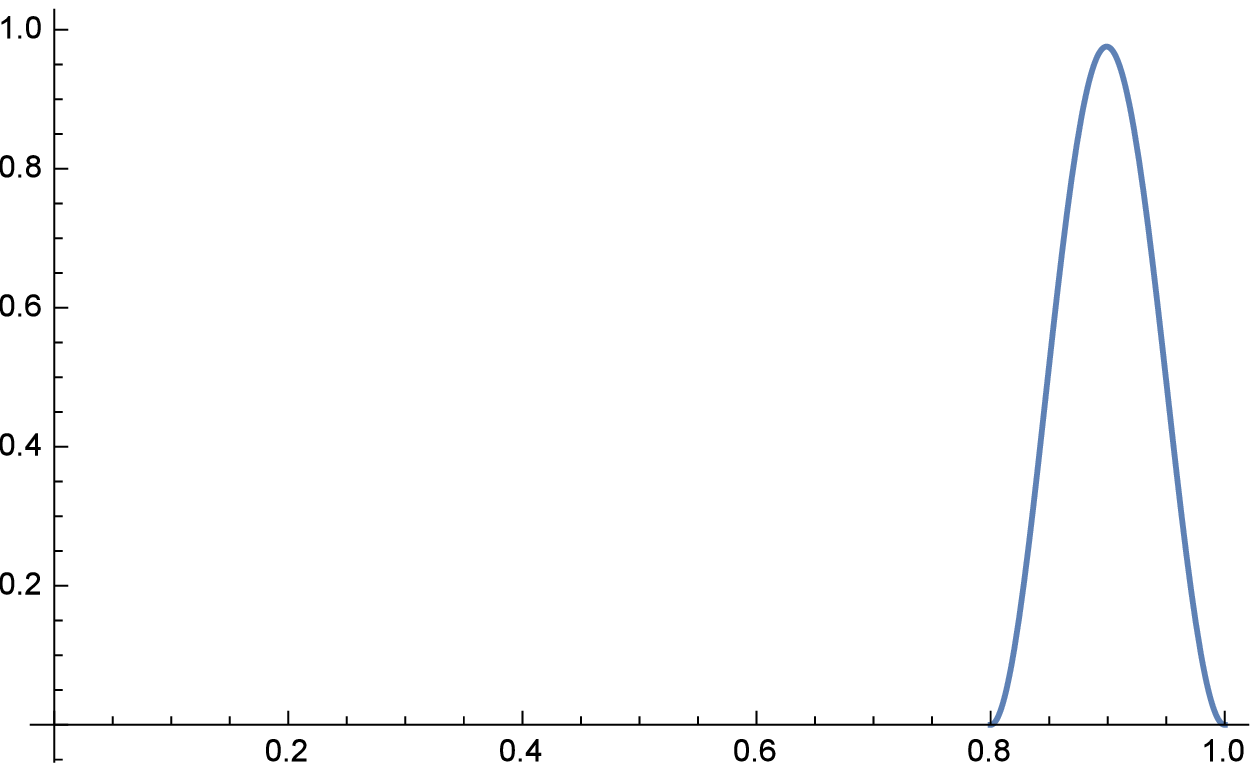} &
	\includegraphics[width=0.28\textwidth]{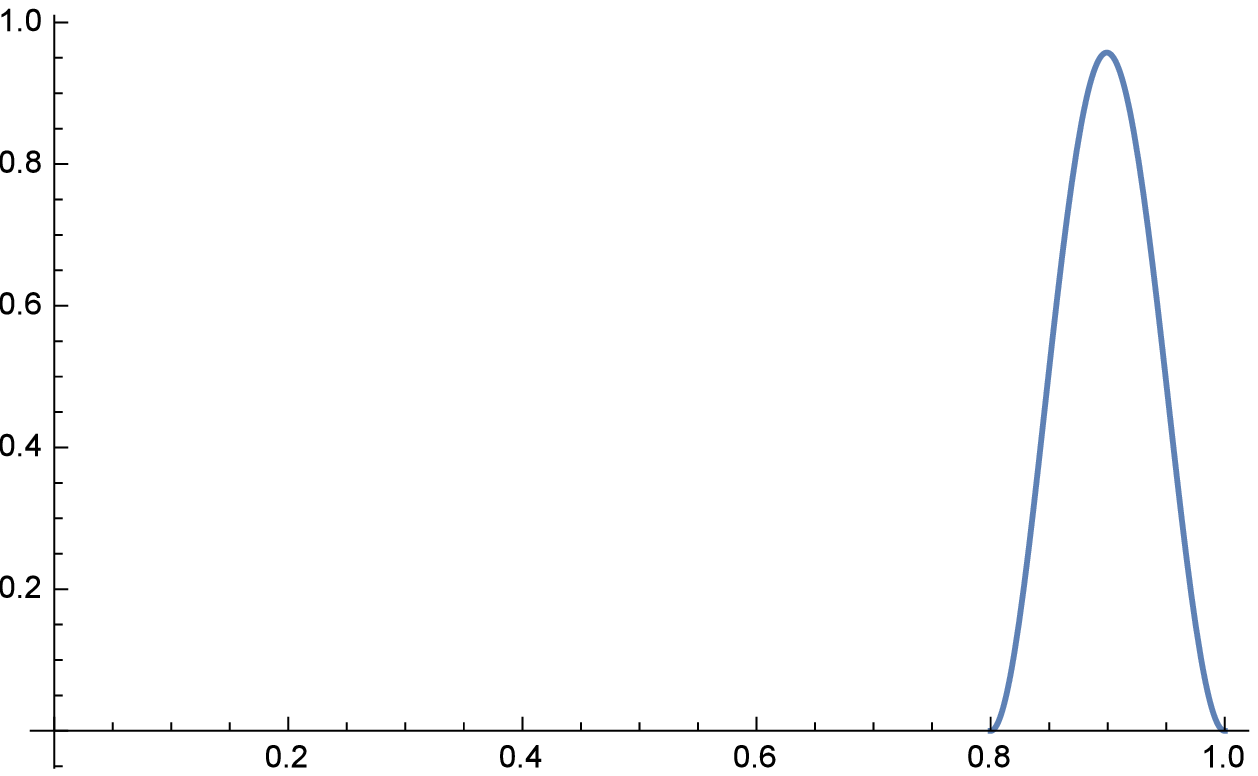}\\
	(g) $k=5$, $a=0.8$ & (h) $k=11$, $a=0.8$ & (i) $k=13$, $a=0.8$
	\end{tabular}
\caption{The radial part of the eigenfunctions corresponding to $\tau_k(a)$ for various values of $k$ and $a$.}
\label{figura2}
\end{figure}


\section{The rectangle}

Now we turn to the problem
\begin{equation}
\label{bucklingrect}
\begin{cases}
\Delta^2 u =-\lambda \Delta u, & \text{in }\Omega_\ell, \\
u(x,y)=0, & \text{on }\partial \Omega_\ell, \\
u_y(x,y)=0, & \text{on }(0,\pi)\times\{-\ell, \ell\},\\
u_{xx}(x,y)=0, & \text{on }\{0,\pi\}\times(-\ell,\ell),
\end{cases}
\end{equation}
where $\Omega_\ell = (0,\pi)\times(-\ell,\ell)$ with $\ell>0$. We recall that problem \eqref{bucklingrect} is linked with the buckling problem \eqref{bucklingequation} in an annulus $\Omega_a$ as, when $a$ approaches the limit $a\to1^-$, the nodal domains of the first eigenfunctions are annular sectors that, as the annulus shrinks, are close to rectangles and the boundary conditions to impose on the nodal domains are those in \eqref{bucklingrect}.
In order to write the solutions of \eqref{bucklingrect}, we observe that if we perform an odd reflection of the solution $u$ around the $y$-axis, then the reflected function is still in $H^2$, and we can expand it into a Fourier series of the form 
\[ u(x,y)=\sum_{m=1}^{+\infty} h_m(y)\sin{mx}.\]
By regularity theory (see e.g., \cite{ggs}), since $u$ is smooth in $\Omega_\ell$, so are the coefficients $h_m$.
Moreover, the equation reads as
\[ u_{xxxx}+2u_{xxyy}+u_{yyyy} = -\lambda(u_{xx}+u_{yy}),\]
and this leads to the following ordinary differential equation for $h_m$:
\begin{equation}
\label{eqdiffforhm}
h_m^{(iv)}(y) + (\lambda-2m^2)h_m''(y) + m^2(m^2-\lambda)h_m(y)=0.
\end{equation}
The characteristic polynomial of the equation is
\[ t^4+(\lambda-2m^2)t^2 + m^2(m^2-\lambda)=0,\]
and setting $z=t^2$, namely
\[ z^2+(\lambda-2m^2)z+m^2(m^2-\lambda)=0,\]
we see that the roots are $z_1 = m^2 - \lambda$ and $z_2=m^2$. We distinguish three cases, according to the sign of $m^2-\lambda$. In each case, it will be possible to search for even and odd eigenfunctions separately by exploiting the linearity of the differential equation and the identity 
\[ h_m(y) = \frac{h_m(y)+h_m(-y)}{2} + \frac{h_m(y)-h_m(-y)}{2}.\]

\noindent
{\bf The case $\lambda < m^2$.}
The general solution of \eqref{eqdiffforhm} is given by
\[ h_m(y)= a\cosh{my}+b\sinh{my}+c\cosh{\gamma y}+d\sinh{\gamma y},\]
where $\gamma=\sqrt{m^2-\lambda}$. Its derivative is given by
\[ h_m'(y) = am\sinh{my}+bm\cosh{my}+c\gamma\sinh{\gamma y}+d\gamma\cosh{\gamma y}.\]
Even solutions are of the form
\[ h_m(y) = a\cosh{my}+c\cosh{\gamma y},\]
and imposing the boundary conditions, we obtain the relations
\[ \begin{cases}
a\cosh{m\ell} + c\cosh{\gamma \ell} =0,  \\
am\sinh{m\ell}  +c\gamma\sinh{\gamma \ell}  = 0. \end{cases}\]
This system admits a nontrivial solution if and only if
\[\gamma \tanh{\gamma\ell} = m\tanh{m\ell}.\]
However, since the function $t \mapsto t\tanh{t\ell}$ is strictly increasing, we obtain $\gamma=m$ and hence $\lambda=0$, a contradiction.
If we look for odd solutions of the form
\[ h_m(y)= b\sinh{my}+d\sinh{\gamma y},\]
we obtain the relations
\[ \begin{cases}
b\sinh{m\ell}  + d\sinh{\gamma \ell}  =0,  \\
bm\cosh{m\ell}  +d\gamma\cosh{\gamma \ell} = 0. \end{cases}\]
This system admits a nontrivial solution if and only if
\[\frac{\tanh{\gamma\ell}}{\gamma} = \frac{\tanh{m\ell}}{m}.\]
However, since the function $t \mapsto \frac{\tanh{t\ell}}{t}$ is strictly decreasing, we obtain $\gamma=m$ and hence $\lambda=0$, again a contradiction.
In conclusion, no eigenfunction exists in the case $\lambda < m^2$.\\

\noindent
{\bf The case $\lambda = m^2$.}
The general solution of \eqref{eqdiffforhm} is given by
\[ h_m(y)= a\cosh{my}+b\sinh{my}+cy+d,\]
and its derivative by
\[ h_m'(y) = am\sinh{my}+bm\cosh{my}+c.\]
Even eigenfunctions are of the form
\[ h_m(y)= a\cosh{my}+d,\]
and imposing the boundary conditions we obtain $am\sinh{m\ell}=0$, which implies $a=0$ and hence $d=0$.
Odd eigenfunctions are of the form
\[ h_m(y) = b \sinh{my} + cy,\]
and imposing the boundary conditions leads to the system
\[ \begin{cases}
b\sinh{m\ell}  + c\ell  =0,  \\
bm\cosh{m\ell}  +c  = 0. 
\end{cases}\]
The system admits a nontrivial solution if and only if
\[ \sinh{m\ell} - m\ell \cosh{m\ell}=0. \]
However, the function $m \mapsto \sinh{m\ell} - m\ell \cosh{m\ell}$ is strictly decreasing, and is equal to zero if and only if $m=0$, a contradiction. As a consequence, no eigenfunction exists in the case $\lambda=m^2$.\\

\noindent
{\bf The case $\lambda > m^2$.}
This case is the most involved. The general solution of \eqref{eqdiffforhm} is given by
\[ h_m(y)= a\cosh{my}+b\sinh{my}+c\cos{\gamma y}+d\sin{\gamma y},\]
where $\gamma=\sqrt{\lambda - m^2}$, and its derivative is given by
\[ h_m'(y) = am\sinh{my}+bm\cosh{my}-c\gamma\sin{\gamma y}+d\gamma\cos{\gamma y}.\]
Even solutions are of the form
\[ h_m(y)= a\cosh{my}+c\cos{\gamma y},\]
and imposing the boundary conditions, we obtain the relations
\[ \begin{cases}
a\cosh{m\ell}  + c\cos{\gamma \ell}  =0,  \\
am\sinh{m\ell}  -c\gamma\sin{\gamma \ell}  = 0.
\end{cases}\]
This system admits a nontrivial solution if and only if
\[\gamma\tan{\gamma\ell} = -m\tanh{m\ell}.\]
Let $m>0$ be fixed. On each interval of the form $(t_1,t_2):=\displaystyle \left(\frac{1}{\ell}\left(\frac{\pi}{2}+k\pi\right),\frac{1}{\ell}\left(\pi+k\pi\right)\right)$, the function $f$ defined as $f(t):= t\tan{\ell t}$ is negative, strictly increasing, and satisfies $\lim_{t \to t_1} f(t)=-\infty$, $f(t_2)=0$. Therefore, for each $k \in \N$ there exists a unique $\gamma^{(k),even}_m\in \displaystyle \left(\frac{1}{\ell}\left(\frac{\pi}{2}+k\pi\right),\frac{1}{\ell}\left(\pi+k\pi\right)\right)$ such that, for $\gamma=\gamma^{(k),even}_m$, the function
\[ h_m(y)= \cosh{my}-\frac{\cosh{m\ell}}{\cos{\gamma\ell}}\cos{\gamma y}\]
is an even, nontrivial solution of \eqref{eqdiffforhm}. Observe that for $\gamma \in \displaystyle \left(\frac{1}{\ell}\left(k\pi\right),\frac{1}{\ell}\left(\frac{\pi}{2}+k\pi\right)\right)$, where $k \in \N$, no nontrivial even solutions can exist.

Let us now look for odd solutions of the form
\[ h_m(y)= b\sinh{my}+d\sin{\gamma y}.\]
Imposing the boundary conditions, we obtain the relations
\[ \begin{cases}
b\sinh{m\ell}  + d\sin{\gamma \ell}  =0,  \\
bm\cosh{m\ell}  +d\gamma\cos{\gamma \ell}  = 0.
\end{cases}\]
This system admits a nontrivial solution if and only if
\[\frac{\tan{\gamma\ell}}{\gamma} = \frac{\tanh{m\ell}}{m}.\]
Let $m>0$ be fixed. On each interval of the form $(t_1,t_2):=\displaystyle \left(\frac{1}{\ell}\left(\pi+k\pi\right),\frac{1}{\ell}\left(\frac{3}{2}\pi+k\pi\right)\right)$, the function $f$ defined as $f(t):= \frac{\tan{\ell t}}{t}$ is positive, strictly increasing (where defined), and satisfies $f(t_1)=0$, $\lim_{t \to t_2} f(t)=+\infty$. Therefore, for each $k \in \N$ there exists a unique $\gamma^{(k),odd}_m\in \displaystyle \left(\frac{1}{\ell}\left(\pi+k\pi\right),\frac{1}{\ell}\left(\frac{3}{2}\pi+k\pi\right)\right)$ such that, for $\gamma=\gamma^{(k),odd}_m$, the function
\[ h_m(y)= \sinh{my}-\frac{\sinh{m\ell}}{\sin{\gamma\ell}}\sin{\gamma y}\]
is an odd, nontrivial solution of \eqref{eqdiffforhm}. Observe that for $\gamma \in \displaystyle \left(\frac{1}{\ell}\left(\frac{\pi}{2}+k\pi\right),\frac{1}{\ell}\left(\pi+k\pi\right)\right)$, where $k \in \N$, no nontrivial odd solutions can exist. Moreover, if $\gamma \in \displaystyle \left(0, \frac{\pi}{2\ell} \right)$, one has
\[ \frac{\tanh{m\ell}}{m}< 1 < \frac{\tan{\gamma\ell}}{\ell},\]
and therefore no nontrivial odd solution can exist in this case.

Let us study more closely the function
\[ h_m(y)= \cosh{my}- \frac{\cosh{m\ell}}{\cos{\gamma\ell}}\cos{\gamma y}\]
when $\gamma \in \displaystyle \left(\frac{\pi}{2\ell},\frac{\pi}{\ell}\right)$. We want to prove that $h_m \geq 0$. Notice that $\cos{\gamma \ell}<0$, therefore if $y \in \left[0,\frac{\pi}{2\gamma}\right]$, we clearly have $h_m(y)>0$. If $y \in \displaystyle \left(\frac{\pi}{2\gamma},\ell\right]$, it is enough to show that
\[ \frac{\cosh{my}}{\cos{\gamma y}} \leq \frac{\cosh{m \ell}}{\cos{\gamma \ell}},\]
since $\cos{\gamma y}<0$. Let us consider the function $f$ defined as
\[ f(t) := \frac{\cosh{mt}}{\cos{\gamma t}}.\]
We have
\[ f'(t) = \frac{m\sinh{mt}\cos{\gamma t} + \gamma\cosh{mt}\sin{\gamma t}}{(\cos{\gamma t})^2}.\]
Hence it is enough to prove that
\[ g(t) := m \tanh{mt} + \gamma\tan{\gamma t} \leq 0\] on $\displaystyle\left(\frac{\pi}{2\gamma}\ell,\ell\right]$. On this interval $g$ is increasing, and $g(\ell)=0$. Therefore, $g(t)\leq 0$, and $f'(t)\geq 0$, which implies the claim.

We can summarize the previous findings in the following

\begin{thm}
\label{efrect}
The problem
\[ \begin{cases}
\Delta^2 u=-\lambda \Delta u, & \text{in }\Omega, \\
u(x,y)=0, & \text{for }(x,y)\in\partial \Omega, \\
u_y(x,y)=0, & \text{for }y \in\{-\ell, \ell\}, \\
u_{xx}(x,y)=0, & \text{for }x \in \{0,\pi\}.
\end{cases}\]
admits a sequence of eigenfunctions $u_{k,m}$, $k$, $m \in \N^*$, of the form
\[ u_{k,m}(x,y) = h_{k,m}(y)\sin{mx},\]
associated to the eigenvalues $\lambda_{k,m}=m^2+\gamma_{k,m}^2$, respectively, where
\[ \gamma_{k,m} \in \left(\frac{k\pi}{2\ell},\frac{(k+1)\pi}{2\ell} \right).\]
The function $h_{k,m}$ is a solution of the problem
\[ \begin{cases}
h^{(iv)}(y) + (\lambda-2m^2)h''(y) + m^2(m^2-\lambda)h(y)  =  0, \\ 
h(-\ell) = h(\ell)  =  0, \\ 
h'(-\ell) = h'(\ell)  =  0,
\end{cases}\]
for $\lambda=\lambda_{k,m}$. Moreover:
\begin{enumerate}
\item[(a)] if $k$ is odd, then $h_{k,m}$ is even, and if $k$ is even, then $h_{k,m}$ is odd.
\item[(b)] If $k=1$, then $h_{k,m}$ is strictly positive.
\end{enumerate}
\end{thm}

We now study the dependence of the first eigenvalue $\lambda_1(\Omega_\ell)$ upon the parameter $\ell$. By Theorem \ref{efrect}, it is clear that $\lambda_1(\Omega_\ell) = \lambda_{1,m}$ for some $m \in \N^*$. Therefore, it is essential to analyze the function defined on $(0,+\infty)$ as
\[ m \mapsto \lambda_{1,m} = m^2+\gamma_{1,m}^2.\]
Without loss of generality, since we can recover any case by scaling, from now on we will consider $\ell=1$.

\begin{lem}
The function $\Phi:(0,+\infty) \to \left(\frac{\pi}{2},\pi\right)$  defined as $\Phi(m)= \gamma_{1,m}$ is of class $C^\infty$, strictly decreasing, and satisfies
\[ \lim_{m \to 0^+} \Phi(m) = \pi, \qquad \lim_{m \to +\infty} \Phi(m) = \frac{\pi}{2}.\]
\end{lem}
\begin{proof}
Let us define the functions $f$ and $g$ as $f(\gamma)= \gamma \tan{\gamma}$ and $g(m)=m\tanh{m}$. It holds
\[ f'(\gamma) = \tan{\gamma} + \frac{\gamma}{\cos^2{\gamma}}.\]
It is clear that $f'(\gamma) > 0$ on $\left(\frac{\pi}{2},\pi\right)$, and therefore $f$ admits a reciprocal function $f^{-1}$ of class $C^\infty$. Observing that $\Phi(m)=f^{-1}(g(m))$, $\Phi$ is then of class $C^\infty$. Moreover, $g'(m) = \tanh{m} + \frac{m}{\cosh^2{m}} > 0$ for every $m>0$. Therefore,
\[\Phi'(m) = (f^{-1})'(g(m)) \cdot g'(m) = \frac{g'(m)}{f'(f^{-1}(g(m))} < 0.\]
\end{proof}

\begin{thm}
There exists $m^* \in (0,+\infty)$ such that $\lambda_{1,m^*} = \min_{m > 0} \lambda_{1,m}$. 
\end{thm}
\begin{proof}
It holds $\lim_{m \to 0^+} \lambda_{1,m} = \pi^2$ and $\lim_{m \to +\infty} \lambda_{1,m} = +\infty$. Since the function can be extended to a continuous function on $[0,+\infty)$, there exists $m^* \in [0,+\infty)$ such that $\lambda_{1,m^*} = \min_{m \geq 0} \lambda_{1,m}$. By direct computation, it holds $\gamma_{1,1} \simeq 2.8833,$ and hence $\lambda_{1,1} \simeq 9.3134 < \pi^2 \simeq 9.8696$. Therefore, $m^* > 0$ and the claim is proved.
\end{proof}

We observe that $\lambda_{1,m}$ is an eigenvalue if and only if $m \in \N^*$. If $\lambda_{1,m}$ is an eigenvalue for $\Omega_\ell$ then, by scaling, $\frac{1}{\varepsilon^2}\lambda_{1,\frac{m}{\varepsilon}}$ is an eigenvalue for $\Omega_{\varepsilon \ell}$, provided $\frac{m}{\varepsilon} \in \N^*$. This leads to the following proposition.

\begin{thm}
\label{crescendo}
Let $\lambda_1(\Omega_\ell)$ be the first eigenvalue in $\Omega_\ell := (0,\pi)\times (-\ell,\ell)$. Then there exists a sequence $\ell_k \to 0$ such that there exists an eigenfunction associated to $\lambda_1(\Omega_{\ell_k})$ with $k$ nodal domains.
\end{thm}


\section{Conclusions}

The present work is a first step towards a better understanding of the buckling eigenvalue problem in an annulus, and several open questions deserve to be further investigated. Even though an analytical proof is still missing, numerical evidence suggests that the number of angular nodal domains of the first eigenfunction is monotonically increasing with respect to the inner radius, and tends to infinity as the annulus shrinks to the circumference. The positivity of the radial part of the first eigenfunction also seems to hold true according to the numerics, but this still remains an open problem.

A general question of deep interest is to determine necessary and sufficient conditions on the domain $\Omega$ for the positivity of the first buckling eigenfunction. Similarly to the case of the clamped plate eigenvalue problem \eqref{bilaplacian}, convexity is not sufficient to guarantee positivity, and neither is smoothness of the boundary. In this paper we give a partial answer to this question, as we show that annuli always have a sign-changing first eigenfunction, and we studied particularly the case of a shrinking annulus where the minimal number of nodal domains is even expected to diverge. A natural question which arises from our investigation is whether there exists a non-simply connected domain with positive first buckling eigenfunction. Moreover, it is interesting to observe that annuli with sufficiently big inner radius provide an example of domains with a positive first clamped eigenfunction, but with a sign-changing first buckling eigenfunction. One might wonder whether the opposite holds true for some other domain; should this be the case, positivity of clamped and buckling eigenfunctions would then be completely unrelated.

As the annuli $\Omega_a$ are shrinking towards the circumference, one may expect some type of convergence for eigenvalues or eigenfunctions to a limiting problem on $\mathbb S^1$, but this does not seem to be the case since even the normalized eigenvalues $\lambda(\Omega_a)|\Omega_a|$ diverge. Observing the different eigenbranches, we see that the eigenfunctions are naturally forced to localize, but we remark that this is not enough to expect any limiting behavior. It is interesting to note that, as the first eigenvalue keeps on changing branch as $a\to1^-$, the number of nodal domains for an associated eigenfunction increases. We expect this quantity to diverge, and we have numerical evidence for this ansatz (see Table \ref{tabellanumerica}), but an analytical proof of this is still missing. Indeed, if we chop the shrinking annulus into its nodal domains and study them as ``close to rectangles'', the analysis leads towards the same direction. In particular, as the first eigenfunction will always have at least a diametrical nodal line, if we think the semi-annulus as a rectangle, then Theorem \ref{crescendo} proves precisely that the number of nodal domain of the first eigenvalue will diverge.
We stress the fact that the clamped plate problem \eqref{bilaplacian} has a totally different behavior, since in that case the first eigenfunction is simple and does not change sign in the regime $a\to1^-$. This highlights that, while the two problems present a lot of similarities, they are significantly different and this in particular motivates the need for further investigations on both problem \eqref{bilaplacian} and \eqref{bucklingequation}.

Finally, the extension of our results to higher dimensional domains still remains open. Multidimensional annuli will be still interesting to study as the eigenfunctions are explicitly written in terms of Bessel functions and spherical harmonics as we do in Section \ref{sec2}, but now the equation involves ultraspherical Bessel functions. Moreover, if the dimension is greater than two, then singlets have zero $H^2$-capacity, which implies the spectral convergence for $a\to0^+$, i.e., when the annulus approaches the ball. In particular it follows that, in higher dimension, the annulus with an inner radius close to zero must then have a simple first eigenvalue. However, the study of the sign of an associated eigenfunction is even more complicated as it involves the study of higher order ultraspherical Bessel functions.



\section*{Acknowledgements}

This work was partially supported by the Funda\c c\~{a}o para a Ci\^{e}ncia e a Tecnologia
(Portugal) through the program ``Investigador FCT'' with reference IF/00177/2013 and the project {\it Extremal spectral quantities and related problems} (PTDC/MAT-CAL/4334/2014). A significant part of the research in this paper was carried out while the first author held a post-doctoral position at the University of Lisbon within the scope of this project. 
The authors express their gratitude to the University of Lisbon and the University of Aix-Marseille for the hospitality that helped the development of this paper. 
The first author is a member of the Gruppo Nazionale per l'Analisi
Matematica, la Probabilit\`a e le loro Applicazioni (GNAMPA) of the Istituto Naziona\-le di Alta Matematica (INdAM).


\bigskip

\end{document}